\documentclass{amsart}
\usepackage{tabularx}
\usepackage{enumitem,amsmath,amsfonts,amssymb,xcolor,latexsym,euscript}
\usepackage[numbers]{natbib}

\newtheorem{lemma}{\bf Lemma}[section]

\newtheorem{defi}[lemma]{\bf Definition}
\newtheorem{prop}[lemma]{\bf Proposition}
\newtheorem{thm}[lemma]{\bf Theorem}

\newcommand{\GL}{{\operatorname{GL}}}
\newcommand{\SL}{{\operatorname{SL}}}

\newcommand{\PGammaL}{{\operatorname{P\Gamma L}}}
\newcommand{\PSL}{{\operatorname{PSL}}}

\DeclareMathOperator{\Sym}{Sym}
\DeclareMathOperator{\Aut}{Aut}
\DeclareMathOperator{\Out}{Out}
\DeclareMathOperator{\soc}{soc}
\DeclareMathOperator{\Stab}{Stab}

\DeclareMathOperator{\STS}{STS}
\DeclareMathOperator{\KTS}{KTS}

\input xy
\xyoption{all}

\title[]{On pyramidal groups of prime power degree}

\author{Xiaofang Gao} 
\address{Xiaofang Gao. Departamento de Matem\'atica, Universidade de Bras\'ilia, Campus 
Universit\'ario \\ Darcy Ribeiro, Bras\'ilia-DF, 70910-900, Brazil. \newline
ORCID:  https://orcid.org/0000-0001-8106-7941}
\email{gaoxiaofang2020@hotmail.com}

\author{Martino Garonzi}
\address{Martino Garonzi. Departamento de Matem\'atica, Universidade de Bras\'ilia, Campus 
Universit\'ario Darcy Ribeiro, Bras\'ilia-DF, 70910-900, Brazil. \newline
ORCID: https://orcid.org/0000-0003-0041-3131}
\email{martino@mat.unb.br, mgaronzi@gmail.com}

\thanks{The first author acknowledges the support of the CAPES PhD
fellowship and the NSF of China - Grant number 12161035. 
The second author acknowledges the support of Conselho Nacional 
de Desenvolvimento Cient\'ifico e Tecnol\'ogico (CNPq), Universal
- Grant number 402934/2021-0.}

\date{}


\begin{document}

\begin{abstract}
A Kirkman Triple System $\Gamma$ is called $m$-pyramidal if there exists a subgroup $G$ of the automorphism group of $\Gamma$ that fixes $m$ points and acts regularly on the other points. Such group $G$ admits a unique conjugacy class $C$ of involutions (elements of order $2$) and $|C|=m$. We call groups with this property $m$-pyramidal. We prove that, if $m$ is an odd prime power $p^k$, with $p \neq 7$, then every $m$-pyramidal group is solvable if and only if either $m=9$ or $k$ is odd. The primitive permutation groups play an important role in the proof. We also determine the orders of the $m$-pyramidal groups when $m$ is a prime number.
\end{abstract}

\maketitle

\setlength{\parskip}{2mm}

\section{Introduction}

A Steiner triple system of order $v$, briefly $\STS(v)$, is a pair $(V,\mathfrak{B})$ where $V$ is
a set of $v$ points and $\mathfrak{B}$ is a set of $3$-subsets (blocks) of $V$ with the property that any two
distinct points are contained in exactly one block. A Kirkman triple system of order $v$, briefly $\KTS(v)$, is an $\STS(v)$ together with a resolution $R$ of its block-set $\mathfrak{B}$, that is a partition of $\mathfrak{B}$
into classes (parallel classes) each of which is, in its turn, a partition of the point-set $V$. It has been known since the mid-nineteenth century that a $\STS(v)$ exists if and only if $v \equiv 1$ or $3 \mod 6$ \cite{Kirkman}. The analogous result for KTSs has been instead obtained more than a century later \cite{RW}: a $\KTS(v)$ exists if and only if $v \equiv 3 \mod 6$. These structures fall into the broader category of combinatorial designs. A combinatorial design $D$ is said to be $m$-pyramidal if there exists a subgroup of $\Aut(D)$ which fixes $m$ points and acts sharply transitively on all the other points. The existence problem for $3$-pyramidal STSs was completely settled in \cite{BRT}: there
exists a 3-pyramidal $\STS(v)$ if and only if $v \equiv 7, 9, 15 \mod 24$ or $v \equiv 3, 19 \mod 48$. Bonvicini, Buratti,  Garonzi,  Rinaldi and Traetta \cite{SMG} gave a necessary condition for the existence of a $3$-pyramidal $\KTS(v)$ and provided some difference methods to construct $3$-pyramidal Kirkman triple systems.

An automorphism of a Kirkman triple system is a permutation of the point set that sends blocks to blocks and parallel classes to parallel classes. Let $m \geqslant 1$ be an integer. A Kirkman triple system is called $m$-pyramidal if there is a subgroup $G$ of the automorphism group of the Kirkman triple system which fixes $m$ points and acts regularly on the other points. In this case, we say that the $\KTS$ is realized ``under'' $G$. Call $\infty_i$ the fixed points, for $i=1,\ldots,m$. Clearly, we may identify the vertices of the $\KTS$ with $V = G \cup \{\infty_1, \ldots, \infty_m\}$. $G$ acts on $V$ by right multiplication on the elements of $G$ and by fixing the points $\infty_i$.

\begin{prop}
Assume an $m$-pyramidal $\KTS$ can be realized under a nontrivial group $G$. Then $G$ has precisely $m$ involutions and the involutions of $G$ are pairwise conjugate.
\end{prop}

\begin{proof}
First, we prove that $G$ has precisely $m$ involutions. Note that for each $i \in \{1,\ldots,m\}$ there exists a block $B_i$ passing through $1$ and $\infty_i$, call it $B_i = \{1,x_i,\infty_i\}$. We claim that $x_i \in G$. If it were not the case, then $x_i$ would be a fixed point $\infty_j$ with $j \neq i$, that is $B_i=\{1, \infty_j, \infty _i\}$.  Thus $B_i g = \{g,\infty_j,\infty_i\}$ is a block for any $g\in G$. Since this holds for all $g \in G$ and there is a unique block passing through $\infty_j$ and $\infty_i$, this implies that $G=\{1\}$, a contradiction. Now, $B_ix_i^{-1} = \{x_i^{-1},1,\infty_i\}$ is a block by assumption, hence it is equal to $B_i$ by uniqueness of the block passing through $1$ and $\infty_i$. This implies that $x_i=x_i^{-1}$ hence $x_i$ has order $2$. This gives us $m$ elements of order $2$, namely $x_1,\ldots,x_m$. Now assume $x \in G$ is any element of order $2$ and let $B=\{1,x,y\}$ be the block through $1$ and $x$. Then $Bx=\{x,1,yx\}$ and again, by uniqueness of the block through $1$ and $x$, we have that $yx=y$. Since $x \neq 1$, we deduce that $y$ is a fixed point $\infty_i$, therefore $x=x_i$. This proves that the involutions of $G$ are precisely $x_1,\ldots,x_m$.

Now we prove that $x_1,\ldots,x_m$ are pairwise conjugate in $G$. Let $\mathcal{Q}$ be the parallel class containing the block $B=\{1,x_1,\infty_1\}$. Note that $Bx_1=B$, therefore $B \in \mathcal{Q} \cap \mathcal{Q} x_1$, hence $\mathcal{Q} x_1 = \mathcal{Q}$. Now, let $B_i = \{\infty_i,g_i,h_i\}$ be the block of $\mathcal{Q}$ through $\infty_i$, for a fixed $i \in \{1,\ldots,m\}$. Since $B_i x_1 \in \mathcal{Q} x_1 = \mathcal{Q}$ and the unique block of $\mathcal{Q}$ through $\infty_i$ is $B_i$, we must have $B_i x_1 = B_i$, that is 
$$\{\infty_i, g_ix_1, h_ix_1\}=\{\infty_i, g_i, h_i\}.$$
If $g_ix_1=g_i$ and $h_ix_1=h_i$, then there exist two integers $j, k\in \{1,\ldots, m\}$ with $j\neq i\neq k$ such that $g_i=\infty_j$ and $h_i=\infty_k$, thus $B_i=\{\infty_i, \infty_j, \infty_k\}$. Note that the block $B_i$ is fixed by all $g\in G$, so it belongs to all the parallel classes $\mathcal{Q}g$, $g\in G$, therefore $\mathcal{Q}g\cap \mathcal{Q}\neq \varnothing$ for any $g\in G$, thus $\mathcal{Q}g=\mathcal{Q}$ for any $g\in G$. Let $i\in \{2,\ldots, m\}$ and let $Y=\{x_i, a_i, b_i\}$ be the block in $\mathcal{Q}$ containing $x_i$. Obviously, $Yx_i=\{1, a_ix_i, b_ix_i\}\in \mathcal{Q}x_i=\mathcal{Q}$. This implies that $$\{1,x_1,\infty_1\}=\{1,a_ix_i,b_ix_i\}$$ because $\mathcal{Q}$ is a partition of the point set. But if $a_ix_i=\infty_1$ then $a_i=\infty_1$ is a contradiction, and if $b_ix_i=\infty_1$ then $b_i=\infty_1$ is a contradiction ($\mathcal{Q}$ is a partition of the point set). We deduce that
$g_ix_1=h_i$ and $h_ix_1=g_i$. Note that 
$$B_i g_i^{-1} = \{\infty_i,1,h_ig_i^{-1}\} = \{\infty_i,1,g_ix_1g_i^{-1}\},$$ and this block must equal $\{1,x_i,\infty_i\}$ by uniqueness of the block through $1$ and $\infty_i$. This proves that $g_i x_1 g_i^{-1}=x_i$.
\end{proof}

This motivates the following definition, which is the same we used in \cite{GG}. In this paper, all groups are assumed to be finite.

\begin{defi}[Pyramidal groups]
A group $G$ is called $m$-pyramidal if $G$ has precisely $m$ involutions, which are all conjugate to each other. $G$ is called pyramidal if it is $m$-pyramidal for some $m$.
\end{defi}

For example, if $m$ is any odd integer larger than $1$, the dihedral group $D_{2m}$ with $2m$ elements is $m$-pyramidal. If $q$ is any odd prime power, then $\mathbb{F}^2_{q}\rtimes \SL(2, q)$ is a $q^2$-pyramidal group, nonsolvable for $q>3$ (see Section \ref{primepower}).

The following general question is interesting: if $m$ is an odd positive integer, what are the values of $v$ such that there exists an $m$-pyramidal $\KTS$ on $v$ vertices? Of course, if a $\KTS$ is realized under a group $G$, then $G$ is $m$-pyramidal, however it is important to observe that the converse is not true. For example, if $n$ is any positive integer, then $G=S_3 \times C_{4n+3}$ is a $3$-pyramidal group not associated to any $3$-pyramidal $\STS$ (see \cite[Theorem 3.4 (iii)]{BRT}: such a $\STS$ should have $|G|+3=24n+21$ vertices), and so in particular it is not associated to any $3$-pyramidal $\KTS$. Despite this, it is still very interesting to study the orders of the $m$-pyramidal groups and also the values of the integers $m$ such that there exist $m$-pyramidal groups with some prescribed property. In \cite{GG}, we proved that every $3$-pyramidal group is solvable and we gave a classification of $3$-pyramidal groups. In this paper, we discuss the values of prime powers $m$ for which the $m$-pyramidal groups are all solvable and we determine the orders of the $m$-pyramidal groups when $m$ is a prime number.

Note that if $G$ is a finite group with $m$ involutions, and $m \geqslant 1$, then $m$ is odd. Indeed the size of $\{x \in G\ :\ x \neq x^{-1}\}$ is even, therefore $|G|$ and $|G|-m-1$ are even, so $m$ is odd. Moreover, if $G$ is a $1$-pyramidal group, then a Sylow $2$-subgroup $P$ of $G$ is a cyclic or generalized quaternion group since $P$ has a unique subgroup of order $2$ (see \cite[Theorem 6.11]{MI}) and, if $G$ is nonsolvable, then $P$ is noncyclic (see \cite [Corollary 5.14]{MI}).

In this paper we prove the following result.

\begin{thm}\label{maintheorem}
Let $m$ be a prime power $p^k$ with $p$ an odd prime, $p \neq 7$. Then the following are equivalent.
\begin{enumerate}
    \item Every $m$-pyramidal group is solvable.
    \item $k$ is odd or $m=9$.
\end{enumerate}
\end{thm}

The reason why we need $p \neq 7$ is that, in the proof of Lemma \ref{el2gp}, we need to be able to construct matrices containing at least one $2 \times 2$ block and one $3 \times 3$ block. We believe that our theorem is true also for $p=7$, however we do not have a proof at the moment.

It is natural to ask what happens if $m$ is not a prime power. Is it possible to determine all the values of the odd integer $m$ such that every $m$-pyramidal group is solvable? See the proof of \cite [Lemma 1]{HY} for a classification of the pyramidal nonabelian simple groups.

If an $m$-pyramidal $\KTS$ is realized under a group $G$, then of course the number of vertices is $|G|+m$. For this reason, it is interesting to determine the orders of the $m$-pyramidal groups. Let $m$ be an odd positive integer and let $X_m$ be the set of orders of $m$-pyramidal groups. The set $X_3$ was determined in \cite[Theorem 3.9]{SMG}, it is the set of positive integers of the form $6d$ with $d$ odd or $3 \cdot 2^a \cdot d$ with $a$ even and $d$ odd. In the following result we generalize this result and we concentrate on the case in which $m$ is a prime number. It is interesting to observe that Mersenne primes (i.e. primes of the form $2^n-1$) play a prominent role.

\begin{thm} \label{orders}
Let $m$ be an odd prime number distinct from $7$. If $m$ has the form $2^n-1$ for some integer $n$, set $Y_m = \{2^a \cdot m \cdot d\ :\ n|a,\ d\ \mbox{odd}\}$, otherwise $Y_m = \varnothing$. Write $m-1=2^t \cdot r$ with $r$ odd and let $Z_m = \{2^a \cdot m \cdot d\ :\ 1 \leqslant a \leqslant t,\ d\ \mbox{odd}\}$. Then $X_m = Y_m \cup Z_m$.
\end{thm}

So for example we have
$$\begin{array}{l}
X_5 = \{10d\ :\ d\ \mbox{odd}\} \cup \{20d\ :\ d\ \mbox{odd}\}. \\
X_{31} = \{2^a \cdot 31 \cdot d\ :\ 5|a,\ d\ \mbox{odd}\} \cup \{62 d\ :\ d\ \mbox{odd}\}.
\end{array}$$

Our paper is organized as follows. In Section \ref{section_properties} we introduce some lemmas and prove some properties of pyramidal groups. In Section \ref{section_almostsimple} we discuss the so-called minimal almost-simple groups (which are classified in \cite{DanLevy}), which we need to prove the main theorem.  In Sections \ref{section_mprime}, \ref{primepower} we prove Theorem \ref{maintheorem} and in Section \ref{section_orders} we prove Theorem \ref{orders}.

\section{Pyramidal groups} \label{section_properties}

In this section, we study the basic properties of pyramidal groups. Recall that, if a group of even order is $m$-pyramidal, then $m$ is odd. If $x,y$ are group elements, we will use the notation $x^y$ for $y^{-1}xy$.

\begin{lemma}\label{properties}
    Let $m \geqslant 1$ be an odd integer and let $G$ be an $m$-pyramidal group. Write $|G|=2^a\cdot d$ with $d$ odd and let $H$ be a subgroup of $G$. Denote by $K$ the (characteristic) subgroup of $G$ generated by all the involutions of $G$ and define $C:=C_G(K)$. Then 
    \begin{enumerate}
        \item If $H$ has even order and $HC=G$, then $H$ is $m$-pyramidal.
        \item If $p$ is any prime divisor of $|G|$, $Q$ is a Sylow $p$-subgroup of $C$ and $H=N_G(Q)$, then $H$ is $m$-pyramidal.
        \item If $m$ is a prime, and if $H$ has even order and it contains a Sylow $m$-subgroup of $G$, then $H$ is $m$-pyramidal.
        \item Assume $H \unlhd G$ and that $|H|$ is odd. Let $\varepsilon$ be an involution in $G$ and let $\ell$ be the number of elements $h \in H$ with the property that $h^{\varepsilon} = h^{-1}$. Then $G/H$ is $m/\ell$-pyramidal. In particular, if $G$ is $1$-pyramidal then so is $G/H$ and, if the involutions of $G$ commute pairwise, then $G/H$ is $m$-pyramidal.
        \item If $m$ is a prime number and $H$ is a normal $2$-subgroup of $G$ then $|H|\equiv 1\mod m$. 
    \end{enumerate}
\end{lemma}

\begin{proof}
Part 1. Assume $G$ is $m$-pyramidal and let $H \leqslant G$ with $HC=G$ and $|H|$ even. If $h \in H$ is an involution and $g \in G$ then, writing $g=xc$ where $x \in H$, $c \in C$, we have $h^g=h^{xc}=h^x \in H$. This implies that all the involutions of $G$ belong to $H$ and that they are conjugate in $H$. We deduce that $H$ is $m$-pyramidal.

Part 2. With the help of Frattini's Argument, we have that $HC=G$. Since $Q\leqslant C$, this implies that $K$ centralizes $Q$, and hence $K\leqslant H$, so $|H|$ is even. By item (1), $H$ is $m$-pyramidal.

Part 3. Since $|H|$ is even, $H$ contains at least one involution. Note that $|G:C_G(x)|=m$ for all involutions $x$ of $G$, this implies that $|G/C|$ is a multiple of $m$, and hence a Sylow $m$-subgroup $S$ of $G$ is not contained in $C$. Since $m$ is a prime number, a nontrivial $S$-orbit of the conjugation action on the $m$ involutions must have size $m$. This means that $S$ acts transitively on the involutions, therefore all involutions of $G$ are contained in $H$ and $H$ is $m$-pyramidal.    

Part 4. If $xH$ is an involution of $G/H$, then $x^2 \in H$, hence $o(x) = 2t$ with $t$ odd and $x^t$ is an involution in $G$. Since $Hx$ is an involution in $G/H$, we have $Hx=(Hx)^t=Hx^t$. Hence every involution of $G/H$ has the form $H \varepsilon$ where $\varepsilon$ is an involution of $G$. The involutions of $G$ belonging to the coset $H \varepsilon$ have the form $h \varepsilon$ with $h \in H$ and $(h \varepsilon)^2 = 1$, equivalently $h^{\varepsilon} = h^{-1}$. Therefore $H \varepsilon$ contains $|I_{\varepsilon}|$ involutions, where 
$$I_{\varepsilon} := \{h \in H\ :\ h^{\varepsilon} = h^{-1}\}.$$ 
Since all the involutions are conjugate, the size of $I_{\varepsilon}$ does not depend on $\varepsilon$, let us call it $\ell$. Since each coset of $H$ corresponding to an involution of $G/H$ contains $\ell$ involutions, $G/H$ contains $m/\ell$ involutions. If $m=1$ then of course $m/\ell=1$, in other words, if $G$ is $1$-pyramidal, then $G/H$ is $1$-pyramidal. Assume that the involutions commute pairwise. If there exists an involution $h \varepsilon \in H \varepsilon$ with $h \neq 1$, then $\varepsilon$ commutes with $h \varepsilon$, so $h = (h \varepsilon) \varepsilon$ is an involution, contradicting the fact that $|H|$ is odd. This implies that $\ell=1$, so $G/H$ is $m$-pyramidal.

Part 5. Since $|G:C_G(x)|=m$ for every involution $x$, we have that $m$ divides $|G/C|$ so there is an element $g \in G$ of order a power of $m$ which acts nontrivially on the set of involutions. Since $m$ is a prime number, $g$ does not fix any involution (either all the orbits have size $1$ or there is only one orbit of size $m$). Since $H$ is normal in $G$, the group $\langle g \rangle$ acts on $H$ by conjugation. If $1 \neq h \in H$ is fixed by $g$ then $g$ fixes some power of $h$ which is an involution, so no nontrivial element of $H$ is fixed by $g$. Therefore any orbit of $\langle g \rangle$ acting on $H \setminus \{1\}$ has size a power of $m$ larger than $1$. This implies that $m$ divides $|H|-1$.
\end{proof}

\begin{lemma} \label{quatpyr}
Let $G$ be a finite group whose Sylow $2$-subgroups have only one element of order $2$. Then $G$ is pyramidal.
\end{lemma}
\begin{proof}
Let $x,y$ be two involutions of $G$. There exist two Sylow $2$-subgroups $P,Q$ of $G$ with $x \in P$ and $y \in Q$. Let $g \in G$ be such that $g^{-1}Pg=Q$. Then $Q$ contains $y$ and $g^{-1}xg$. Since $Q$ contains only one involution, it follows that $g^{-1}xg=y$.
\end{proof}

Recall that a finite $2$-group has a unique element of order $2$ if and only if it is cyclic or generalized quaternion (see \cite[Theorem 6.11]{MI}). As usual, we denote by $O(G)$ the maximal normal subgroup of odd order in $G$. 

\begin{lemma} [Theorem 2 of \cite{Suzuki}]\label{quaternionsylow}
 Let $G$ be a finite group whose $2$-Sylow subgroup is a generalized quaternion group. Then $G/O(G)$ has a center of order $2$.   
\end{lemma}

A group $G$ is called primitive if it admits a maximal subgroup $M$ whose normal core $M_G := \bigcap_{g \in G} M^g$ is trivial. The index $|G:M|$ is called a primitivity degree of $G$. This corresponds to saying that the transitive action of $G$ on the set $\Omega$ of cosets of $M$ in $G$ is faithful and also primitive, i.e. it does not preserve any partition of $\Omega$. A good reference for the general properties of primitive groups is \cite{BE}. In the following result, we are concerned with groups in which any two involution do not commute.

\begin{prop} \label{dihedral}
    Let $G$ be a finite group containing precisely $m$ involutions, call them $x_1,\ldots,x_m$. If the following hold
    \begin{enumerate}
        \item $x_i x_j \neq x_j x_i$ for all $i \neq j$,
        \item $G=\langle x_1,\ldots,x_m \rangle$,
    \end{enumerate}
then $G$ is $m$-pyramidal and $G \cong G' \rtimes C_2$ with $|G'|$ odd. Moreover, if $m$ is a prime number, then $G$ is a dihedral group of order $2m$.
\end{prop}

\begin{proof}
First, note that the involutions of $G$ are all conjugate. Indeed, let $P$ be a Sylow $2$-subgroup of $G$ and $z \in P$ an involution belonging to the center of $P$, then $z$ commutes with all the involutions in $P$, so it is the only involution of $P$. By Lemma \ref{quatpyr}, $G$ is $m$-pyramidal. Moreover, $P$ is cyclic or generalized quaternion.

Let $O=O(G)$ be the largest normal subgroup of $G$ of odd order. We claim that $G/O$ has a central involution. If $P$ is cyclic, then $G$ is solvable by \cite[Corollary 5.14]{MI}, therefore a minimal normal subgroup $N/O$ of $G/O$ is, on one hand, an elementary abelian $2$-group, and on the other hand a section of $P$, which is cyclic, so $|N/O|=2$. This means that $G/O$ has a central involution. If $P$ is generalized quaternion, the claim follows from Lemma \ref{quaternionsylow}.
We deduce that there exists an involution $\varepsilon \in G$ such that $O\langle \varepsilon\rangle$ is normal in $G$, and hence the involutions of $G$ are all contained in $O \langle \varepsilon\rangle$. Since $G$ is generated by the involutions, $G = O \langle \varepsilon \rangle$. Since $O\cap \langle \varepsilon\rangle=1$, we have $G=O\rtimes \langle \varepsilon\rangle$ and hence the derived subgroup $G'$ is contained in $O$. If $a\in G$ is written as a product of involutions, $a = y_1 y_2 \ldots y_t$ with $y_s \in \{x_1, \ldots, x_m\}$ for all $s$, then $a\in G'$ if and only if $t$ is even. Indeed, since the involutions are conjugate, any product of two involutions is a commutator, and this implies that if $t$ is even then $a \in G'$. Conversely, if $a \in G'$ then $t$ must be even, since if $t$ is odd then $y_t = y_{t-1} \ldots y_1 a$ belongs to $G'$ since $t-1$ is even, contradicting the fact that $G' \leqslant O$. Therefore, if $\varepsilon \in G$ is an involution, then for every $g \in G \setminus G'$ we have $g \varepsilon \in G'$. This implies that $|G:G'|=2$ hence $G'=O$. Moreover, since $|O|$ is odd, the fact that $G$ is a semidirect product $O \rtimes C_2$ implies that $G$ is solvable.

Now assume that $m$ is a prime number. We will prove, by induction on $|G|$, that $G \cong D_{2m}$. We know that $G = G' \rtimes \langle \varepsilon \rangle$, $o(\varepsilon)=2$. Let $M := C_G(\varepsilon)$. Since $G$ is $m$-pyramidal, $|G:M|=m$, prime, so $M$ is a maximal subgroup of $G$. If the normal core $M_G$ of $M$ in $G$ is trivial then $G$ is primitive of prime degree $m$. Since $G$ is solvable, this implies that $G= V \rtimes M$ where $V \cong C_m$ is a minimal normal subgroup of $G$. On the other hand, since $C_G(V)=V$ we have $G/V \lesssim \Aut(C_m)\cong C_{m-1}$. It follows that $V$ contains $G'$ and the minimality of $V$ implies that $G'=V$. Therefore $G\cong C_m\rtimes C_2\cong D_{2m}$. 

Now assume $M_G \neq \{1\}$ and let $L$ be a minimal normal subgroup of $G$ contained in $M_G$. Let $N:=G'$. Clearly, $|L|$ is odd and $L\neq N$ since otherwise $N=L\leqslant M<G$ and, being $|G:N|=2$, $N=M$ and hence $m=2$, a contradiction. By Lemma \ref{properties}(4), $G/L$ is $1$-pyramidal or $m$-pyramidal. If $G/L$ is $1$-pyramidal, then $x_i x_j \in L$ for any two involutions $x_i$ and $x_j$ of $G$. Since every element of $G'$ is a product of an even number of involutions, $G'\leqslant L$, and since $|G:G'|=2$ we deduce that $G'=L$, contradicting the fact that $L\leqslant M$. Therefore $G/L$ is $m$-pyramidal. We now prove that any two involutions of $G/L$ do not commute. Let $x$ be an involution of $G$ with $x \neq \varepsilon$. If $x \varepsilon L = \varepsilon x L$ then $[x,\varepsilon] \in L \leqslant M = C_G(\varepsilon)$, so 
$$x \varepsilon x = [x,\varepsilon] \varepsilon = \varepsilon [x,\varepsilon] = \varepsilon x \varepsilon x \varepsilon$$ 
hence $\varepsilon$ commutes with $x \varepsilon x$, implying that $x \varepsilon x = \varepsilon$ (because two distinct involutions of $G$ cannot commute), so that $x \varepsilon = \varepsilon x$, a contradiction. This implies that the quotient group $G/L$ satisfies all the conditions of the statement, so by induction we may assume that $G/L \cong D_{2m}$. If $g \in G$ then $L = L^g \leqslant M^g = C_G(\varepsilon^g)$. Since the involutions are all conjugate and generate $G$, we must have $L \leqslant Z(G)$, so $L$ has prime order $p$ (being a minimal normal subgroup of $G$). Since $|G:L|=2m$, we deduce that $|G|=2mp$ and $N=G'$ is abelian because $N/L\cong C_m$ and $L\leqslant Z(N)$. 
Therefore $N\cong C_m\times C_p$ ($m\neq p$), $C_m\times C_m$ or $C_{m^2}$, and then $G\cong (C_m\times C_p)\rtimes C_2$, $(C_m\times C_m)\rtimes C_2$ or $C_{m^2}\rtimes C_2\cong D_{2m^2}$. Since $Z(G) \neq \{1\}$ and $G$ is generated by involutions, we deduce a contradiction.
\end{proof}

\begin{lemma} [Theorem 1 of \cite{Daniel}]\label{dihedralsylow}
Let $G$ be a finite group with dihedral Sylow $2$-subgroups. Then $G/O(G)$ is isomorphic to one of the following.
\begin{enumerate}
 \item A subgroup of $\PGammaL(2,q)$ containing $\PSL(2,q)$, $q$ odd.
 \item The alternating group $A_7$.
 \item A Sylow $2$-subgroup of $G$.
\end{enumerate}
\end{lemma}

We now want to study pyramidal groups in which the involutions commute pairwise. First, we list some preliminary results.

\begin{lemma} \label{2tr}
Let $V=\mathbb{F}_p^n$ and let $H$ be a subgroup of $\GL(V)$. The semidirect product $G=V \rtimes H$ acts naturally (and transitively) on the right cosets of $H$ in $G$. This action is $2$-transitive if and only if $H$ acts transitively on $V \setminus \{0\}$.
\end{lemma}
\begin{proof}
Let $\Omega := \{Hv\ :\ v \in V\}$ be the set on which $G$ acts. Saying that this action is $2$-transitive is equivalent to saying that there exists $\alpha \in \Omega$ such that the action of $\Stab_G(\alpha)$ on $\Omega \setminus \{\alpha\}$ is transitive. Equivalently, $H$ acts transitively on $\Omega \setminus \{H\}$. In other words, for all $v,w \in V$ with $v \neq 0 \neq w$, there exists $h \in H$ such that $Hvh=Hw$. Equivalently, for every two nonzero vectors $v,w \in V$ there exists $h \in H$ such that $vhw^{-1} \in H$. This condition is equivalent to $h^{-1}vhw^{-1} \in H$ and, since $V \unlhd G$ and $V \cap H=\{1\}$, this is equivalent to saying that $v^h=w$.
\end{proof}

\begin{thm}[Zsigmondy's Theorem \cite{KZ}]
Let $a$ and $n$ be integers greater $1$. Then there exists a prime divisor $q$ of $a^n-1$ such that $q$ does not divide $a^j-1$ for all $0<j<n$, except exactly the following cases:
\begin{enumerate}
\item $n=2$, $a=2^s-1$, where $s\geqslant 2$.
\item $n=6$, $a=2$.
\end{enumerate}
\end{thm}

The following lemma shows, in particular, that prime powers of the form $2^n-1$ are actually prime numbers. In other words, Mersenne prime powers are prime.

\begin{lemma} \label{Mersenne}
If $p$ is a prime number such that $p^k = a^n-1$ for some integers $k \geqslant 1$ and $a,n > 1$, then one of the following holds.
\begin{enumerate}
    \item $(p,k,a,n)=(2,3,3,2)$,
    \item $a=2$, $k=1$ and $n$ is a prime number.
\end{enumerate}
\end{lemma}
\begin{proof}
Assume $n \not \in \{2,6\}$. Then Zsigmondy's theorem implies that $a^n-1$ has a primitive prime divisor, which of course must be equal to $p$. It follows that $p$ does not divide $a^j-1$ for all $j$ with $1 \leqslant j < n$. However, $p^k=a^n-1$ is of course divisible by $a-1$, so since $p$ does not divide $a-1$ we deduce that $a=2$. Since if $n=6$ then $a=2$, this argument shows that either $a=2$ or $n=2$. If $n=2$ then $p^k = (a-1)(a+1)$ implies that $(p,k,a,n)=(2,3,3,2)$. Now assume that $a=2$.

We have $p^k = 2^n-1$ hence $p$ is odd. We will prove that $k=1$ and that $n$ is a prime number. If $n=2$ the claim is obvious, now assume $n \geqslant 3$. 
Since $p$ is odd, it is congruent to $1$, $3$, $5$ or $7$ modulo $8$, so $p^2 \equiv 1 \mod 8$. On the other hand, $p^k = 2^n-1 \equiv -1 \mod 8$ since $n \geqslant 3$. This implies that $k$ is odd, hence we have a factorization
$$2^n = p^k+1 = (p+1) (p^{k-1}-p^{k-2}+\ldots-p+1)$$
The second factor is a sum of $k$ odd integers, so it is odd. Since it divides $2^n$, it must be equal to $1$, so that $2^n=p+1$ implying $k=1$. We now prove that $n$ is a prime number. Consider a factorization $n=rs$ with $r,s$ positive integers and $s > 1$. We will prove that $r=1$. Note that
$$p = 2^n-1 = 2^{rs}-1 = (2^r-1)(1+2^r+\ldots+2^{r(s-1)})$$
implying that one of the two factors must equal $1$. Since $s > 1$, the second factor is larger than $1$, so $2^r-1=1$, i.e. $r=1$.
\end{proof}

As usual, we denote by $\Phi(G)$ the Frattini subgroup of $G$, which is the intersection of all the maximal subgroups of $G$.

\begin{lemma} \label{el2gp}
Assume that $m$ is a prime power $p^k$ where $p$ is an odd prime distinct from $7$, $G$ is an $m$-pyramidal group and the subgroup $K$ of $G$ generated by the involutions is an elementary abelian $2$-group of size $2^n$, so that $2^n-1=m$. Then $m$ is a prime number and $G/C_G(K)$ is isomorphic to one of $C_m$, $C_m \rtimes C_n$, where in the second case the action is given by the Frobenius automorphism of a field $F$ of size $2^n$ acting naturally on $F \setminus \{0\}$.
\end{lemma}

\begin{proof}
Lemma \ref{Mersenne} implies that $m$ and $n$ are actually prime numbers and $n\neq 3$. Clearly, $K$ is a minimal normal subgroup of $G$ since $K\cong C^n_2\unlhd G$ and all involutions of $G$ are conjugate. Moreover we have a faithful action by conjugation of 
$N_G(K)/C_G(K)=G/C_G(K)$ on $K$, and hence $G/C_G(K)$ is isomorphic to a subgroup $H$ of $\Aut(K) \cong \GL(n,2)$ and $H$ is irreducible because the nontrivial elements of $K$ are pairwise conjugate in $G$. 

We have $K\cong \mathbb{F}^n_2$, $n\neq 3$ prime, $m=2^n-1\neq 7$ prime. Since $G$ is an $m$-pyramidal group, all elements of $K\setminus \{0\}$ form the single conjugacy class of involutions in $G$ and $H$ acts transitively on $K\setminus\{0\}$. Set $$X:=K\rtimes H,$$ where $\rtimes$ is defined with respect to the action of $H\cong G/C_G(K)$. Since $H$ acts transitively on $K\setminus \{0\}$, we have that $X$ is a $2$-transitive group, by Lemma \ref{2tr}. Let $kh\in C_X(K)$ where $k\in K$, $h\in H$, then $v^{kh}=v^h=v$ for any involution $v$ of $K$. Since $H$ acts transitively on $K\setminus \{0\}$, $h=1$, it follows that $C_X(K)=K$. Obviously, $K$ is a minimal normal subgroup of $X$.
We claim that $K$ is the unique minimal normal subgroup of $X$. Indeed, suppose that $L\neq K$ is a minimal normal subgroup of $X$, then $[L, K]=\{1\}$, and hence $L< C_X(K)=K$, contradicting the minimality of $K$. This proves that $K$ is the unique minimal normal subgroup of $X$, i.e., $\soc(X)=K$. Therefore, $X$ is a $2$-transitive affine group of degree $2^n$.   
Consider \cite[Table 7.3]{PJC}, which is the classification of $2$-transitive affine groups. Since $n$ is prime, we are only concerned with the first line of the table, in which the degree $q^d$ can be interpreted in two different ways: $q=2^n$, $d=1$ or $q=2$, $d=n$. The second case gives $H= \GL(n, 2)$. The first case gives $$C_m \cong \GL(1, 2^n) \leqslant H \leqslant \Gamma L(1,2^n).$$ Since $n$ is prime, there are only three possibilities for $H$: namely $\GL(n, 2)$, $C_m$ and $C_m \rtimes C_n$ where, in the third case, the action of $C_n$ on $C_m$ is precisely the natural action of the Frobenius automorphism of order $n$, that is $\phi: x \mapsto x^2$, on the set $F^{\ast} = F \setminus \{0\}$ where $F$ is the finite field of size $2^n$. In other words $$C_m \rtimes C_n = F^{\ast} \rtimes \langle \phi \rangle.$$ Also, note that $\GL(n, 2)$ and $C_m \rtimes C_n$ are the same group if $n=2$.

We are left to prove that $G/C \not \cong \GL(n, 2)$ for $n>3$, where $C=C_G(K)$. We assume that $G$ is a group of minimal order with the following three properties: $G$ is $m$-pyramidal, the involutions of $G$ commute pairwise and $G/C \cong \GL(n, 2)$. Assume that there exists a maximal subgroup $M$ of $G$ such that $C \nleqslant M$, then $M<MC\leqslant G$, thus $MC=G$ by the maximality of $M$. Since 
$$\GL(n,2) \cong G/C = MC/C \cong M/(M\cap C)$$ we have that $|M|$ is even. Lemma \ref{properties}(1) implies that $M$ is $m$-pyramidal, that is $K\leqslant M$ and all involutions of $M$ are conjugate. Note that $$M/C_M(K)=M/(C\cap M)\cong G/C\cong \GL(n, 2)$$ and $|M|<|G|$, moreover the involutions of $M$ commute pairwise because they are involutions in $G$. This contradicts the minimality of $|G|$. Therefore $C$ is contained in every maximal subgroup of $G$, in other words $C \leqslant \Phi(G)$. Since $n>3$, $$G/C \cong \GL(n, 2) \cong \PSL(n,2)$$ is a simple group, so $C$ is a maximal normal subgroup of $G$ and hence $C=\Phi(G)$. 
We claim that $K$ is the unique minimal normal subgroup of $G$. Indeed, suppose that $L\neq K$ is a minimal normal subgroup of $G$, then $L$ does not contain $K$, so it does not contain involutions, implying that $|L|$ is odd. Note that $L\leqslant C$ because otherwise $CL=G$ by maximality of $C$ as a normal subgroup of $G$, and then $$\GL(n,2) \cong G/C \cong L/(L\cap C),$$ this contradicts the fact that $|L|$ is odd. In this case $\overline{K} \cong K$ is a minimal normal subgroup of $\overline{G}$ where, for $R \leqslant G$, we define $\overline{R} := RL/L$. The fact that $L \cap K = \{1\}$ easily implies that $C_{\overline{G}}(\overline{K}) = \overline{C}$. Moreover, $$\GL(n,2) \cong G/C \cong \overline{G}/\overline{C}.$$ By Lemma \ref{properties}(4), $\overline{G}$ is an $m$-pyramidal group in which the involutions commute pairwise. Since $L \neq \{1\}$, this contradicts the minimality of $|G|$. This proves that $K$ is the unique minimal normal subgroup of $G$, in particular $O(G)=\{1\}$.
This implies that $\Phi(G)=C$ is a $2$-group. Indeed, if it were not, then it would have a Sylow subgroup of odd order, which would be normal in $G$ because $\Phi(G)$ is nilpotent, contradicting $O(G)=\{1\}$.

We are in the following situation: $G/C \cong \GL(n, 2)$, $n>3$, $K = \langle x_1,\ldots,x_m \rangle \cong \mathbb{F}_2^n$ is the unique minimal normal subgroup of $G$ and $C=C_G(K)=\Phi(G)$ is a $2$-group. Let $\gamma:G/C \to \GL(n,2)$ be the canonical map, which by assumption is an isomorphism.
We follow Dempwolff \cite[page 2]{UD}. Let $\{v_1,\ldots,v_n\}$ be an $\mathbb{F}_2$-basis of $K$ and let $\tau\in G\setminus C$ act in the following way: $v^{\tau}_i=v_{i-1}+v_i$ for all $i\leqslant n-3$ even and $v^{\tau}_i=v_i$ for all other values of $i$. Observe that $\tau^2 \in C$ and, with respect to the given basis, we have
$$\gamma(\tau C)= \left( \begin{array}{cccccccc} 1 & 1 & & & & & & \\ 0 & 1 & & & & & & \\ & & \ddots & & & & & \\  & & & 1 & 1 & & & \\ & & & 0 & 1 & & & \\ & & & & & 1 & 0 & 0 \\ & & & & & 0 & 1 & 0 \\ & & & & & 0 & 0 & 1 \end{array} \right).$$

We can decompose $K$ as a direct sum 
$$K=U_1 \oplus U_2\oplus \ldots \oplus U_{(n-1)/2}$$ where $U_i = \langle v_{2i-1},v_{2i} \rangle$ for $i=1,\ldots,(n-3)/2$ and $U_{(n-1)/2} = \langle v_{n-2},v_{n-1},v_n \rangle$. Choose elements $\rho_1 C, \rho_2 C, \ldots, \rho_{(n-3)/2} C$ of order $3$ in $G/C$ such that, for $1 \leqslant i \leqslant (n-3)/2$, $$C_K(\rho_i) = \langle v_1, \ldots, v_{2i-3}, v_{2i-2}, v_{2i+1}, v_{2i+2}, \ldots, v_n\rangle$$
and $\gamma(\rho_i C)$ permutes transitively the nontrivial elements in $\langle v_{2i-1}, v_{2i}\rangle$, and choose $\rho_{(n-1)/2} C$ as an element of order $7$ in $G/C$ with $$C_K(\rho_{(n-1)/2})=\langle v_1, \ldots, v_{n-3}\rangle$$ and which acts irreducibly on $\langle v_{n-2}, v_{n-1}, v_n\rangle$. Note that, for $i=1,2, \ldots, (n-3)/2$ we have

\[
\gamma(\rho_i C) = \begin{pmatrix}
    \begin{matrix} 
       I_{2i-2}
    \end{matrix} & 
    \begin{matrix} 
         & \\
         & 
    \end{matrix} & 
    \begin{matrix} 
         &  \\
         & 
    \end{matrix} \\
    \begin{matrix} 
         &  \\
         & 
    \end{matrix} & 
    \begin{matrix} 
        A
    \end{matrix} & 
    \begin{matrix} 
         &  \\
         & 
    \end{matrix} \\
    \begin{matrix} 
         &  \\
         & 
    \end{matrix} & 
    \begin{matrix} 
         &  \\
         & 
    \end{matrix} & 
    \begin{matrix} 
        I_{n-2i}
    \end{matrix}
\end{pmatrix}
\]
and
\[
\gamma(\rho_{(n-1)/2} C) =\begin{pmatrix}
    \begin{matrix} 
       I_{n-3}
    \end{matrix} & 
    \begin{matrix} 
         & \\
         & 
    \end{matrix} & \\
    \begin{matrix} 
         &  \\
         & 
    \end{matrix} & 
    \begin{matrix} 
        B
    \end{matrix} 
\end{pmatrix}
\]
where $A = \left( \begin{array}{cc} 0 & 1 \\ 1 & 1 \end{array} \right)$ and $B = \left( \begin{array}{ccc} 0 & 0 & 1 \\ 1 & 0 & 0 \\ 0 & 1 & 1 \end{array} \right)$. The element $A$ is an element of $\GL(2,2)$ of order $3$ acting transitively on the nonzero vectors of $\mathbb{F}_2^2$ and $B$ is an element of $\GL(3,2)$ of order $7$ acting transitively on the nonzero vectors of $\mathbb{F}_2^3$. It is easy to check that $\rho_1C, \rho_2C, \ldots, \rho_{(n-1)/2}C$ pairwise commute with each other. We construct the element $$xC=\rho_1\rho_2\dots\rho_{(n-1)/2}C.$$ 
Note that $(\rho_i C)^{\tau C} = \rho_i^2 C$ for $i=1,2,3, \ldots, (n-3)/2$ and $(\rho_{(n-1)/2} C)^{\tau C} = \rho_{(n-1)/2} C$, hence $(x C)^{\tau C} = x^8 C$, and then $x^{\tau}\in x^8C$, thus there exists $c\in C$ such that $x^{\tau}=x^8c$, this implies that $x^{\tau}\in \langle C, x\rangle$. Therefore, the element $\tau$ normalizes $\langle x, C\rangle$ (since $C$ is normal in $G$).
Obviously, $o(xC)=21$ and $o(x) = 2^u \cdot 21$ for some $u$. Up to replacing $x$ with $x^{2^u}$, we may assume that $o(x)=21$. Moreover, $C_K(x)=\{1\}$ and $$\gamma(x C)= \left( \begin{array}{cccccccc} 0 & 1 & & & & & & \\ 1 & 1 & & & & & & \\ & & \ddots & & & & & \\  & & & 0 & 1 & & & \\ & & & 1 & 1 & & & \\ & & & & & 0 & 0 & 1 \\ & & & & & 1 & 0 & 0 \\ & & & & & 0 & 1 & 1 \end{array} \right).$$

\medskip

Now let $H := \langle C,x,\tau \rangle \leqslant G$, $S=\langle x \rangle$, $J := N_H(S)$. Since $\tau$ normalizes $\langle C,x \rangle$, we have $\langle C,x\rangle$ is normal in $H$ and $H/\langle C,x\rangle$ is a $2$-group. Moreover, the order of the quotient group $\langle C,x\rangle/C$ is $21$, it follows that $\langle C,x\rangle$ is solvable since both $C$ and $\langle C,x\rangle/C$ are solvable. Therefore, $H=\langle C,x \rangle \langle \tau\rangle$ is solvable.
Since $C$ is a $2$-group, $S$ is a Hall subgroup of the solvable group $H$, therefore if $h \in H$ then $S^h$ is a Hall subgroup of $\langle C, x\rangle$. This is because $\langle C, x\rangle$ is normal in $H$, and $S$ is contained in $\langle C, x\rangle$. Since $\langle C, x\rangle$ is solvable, there exists $y\in \langle C,x\rangle$ such that $S^h=S^y$, so $S^{hy^{-1}}=S$. Thus $hy^{-1}\in N_H(S)=J$, so that $$H=J\langle C,x\rangle=CJ.$$ 
Note that $\tau \in H = CJ$, let $\tau=c\theta\in CJ$ where $c\in C$ and $\theta\in J$. Since $\tau\in G\setminus C$, we have $1\neq \theta=c^{-1}\tau\in \tau C$. Therefore $\theta C=\tau C$, and we may change the definition of $\tau$ in the beginning of the argument by setting it to be equal to $\theta$. There is no harm in doing this, because all that was used until now has to do only with the action of $\tau$ on $K$, so we have the freedom to change the representative in $\tau C$. Thus we can assume that $\tau \in J$. Since the involutions of $G$ belong to $K$ and $\tau$ does not centralize $K$, $\tau$ is not an involution. Since $C$ is a $2$-group, $\tau$ has order a power of $2$, say $o(\tau)=2^k$ with $k \geqslant 2$, and hence $t := \tau^{2^{k-1}}$ has order $2$ so it belongs to $K$. Therefore $t \in K \cap J$. Since $t$ normalizes $S$, we can write $x^t = x^r$ for some natural number $r$. Since $t \in K$,
$$x^{r-1} = x^r \cdot x^{-1} = x^t \cdot x^{-1} = t^{-1} x t x^{-1}=t^{-1}(xtx^{-1}) \in K$$
and since $x$ has odd order this implies that $x^{r-1}=1$, hence $x^t=x^r=x$, so $t \in C_K(x) = \{1\}$, which implies that $t=1$, a contradiction.
\end{proof}

\section{Almost-simple groups} \label{section_almostsimple}

A group $L$ is called almost-simple if there exists a simple normal subgroup $S$ of $L$ such that $C_L(S)=\{1\}$. In particular $S$ is nonabelian. Under these assumptions, the canonical map $L \to \Aut(S)$ is injective and we may identify $L$ with a subgroup of $\Aut(S)$ containing the image of the canonical map $S \to \Aut(S)$, which is isomorphic to $S$. In particular, the socle of $L$ (the subgroup of $L$ generated by its minimal normal subgroups) is equal to $S$, $\soc(L)=S$. A subgroup $B$ of an almost-simple group $L$ is called faithful if $\soc(L)$ is not contained in $B$, it is called unfaithful if $\soc(L)$ is contained in $B$. The almost-simple group $L$ is called minimal almost-simple if every faithful subgroup of $L$ is solvable. In other words, $L$ is minimal almost-simple if and only if every subgroup of $L$ not containing $\soc(L)$ is solvable. So for example the minimal nonabelian simple groups (i.e. the nonabelian simple groups whose proper subgroups are all solvable) are minimal almost-simple.

Recall that Thompson proved that the minimal simple groups are the following.

\begin{enumerate}
\item $\PSL(2,2^r)$, $r$ any prime.
\item $\PSL(2,3^r)$, $r$ any odd prime.
\item $\PSL(2,p)$, $p>3$ a prime satisfying $p^2+1 \equiv 0 \mod 5$.
\item $\mbox{Sz}(2^r)$, $r$ any odd prime.
\item $\PSL(3,3)$.
\end{enumerate}

A nonabelian simple group $T$ is called ``new'' if it is not minimal simple but it is the socle of a minimal almost-simple group.

\begin{thm}[Theorem 1.20 of \cite{DanLevy}]\label{minimalalmostsimple} 
Let $L$ be an almost-simple group. Then $L$ is minimal almost-simple if and only if its socle $T$ is isomorphic to one of the groups on the following list subject to the indicated conditions.
\begin{enumerate}
\item $\PSL(2,2^r)$, $r$ any prime. $T$ is minimal simple, $\Out(T)=\langle \phi\ |\ o(\phi)=r\rangle$.
\item $\PSL(2,3^r)$, $r$ any odd prime. $T$ is minimal simple, $\Out(T)=\langle \delta,\phi\ |\ o(\delta)=2,\ o(\phi)=r,\ [\delta,\phi]=1\rangle$.
\item $\PSL(2,p)$, $p>3$ a prime satisfying $p^2+1 \equiv 0 \mod 5$. $T$ is minimal simple, $\Out(T)=\langle \delta\ |\ o(\delta)=2\rangle$.
\item $\PSL(2,p)$, $p \geqslant 11$ a prime satisfying $p \equiv \pm 1 \mod 10$ and $L=\Aut(T)$. $T$ is new, $\Out(T)=\langle \delta\ |\ o(\delta)=2 \rangle$.
\item $\PSL(2,p^{2^m})$, $p \geqslant 3$ a prime, $m$ a positive integer and $L/\mbox{Inn}(T) \not \leqslant \langle \phi \rangle$. $T$ is new, $\Out(T)=\langle \delta,\phi\ |\ o(\delta)=2,\ o(\phi)=2^m,\ [\delta,\phi]=1\rangle$.
\item $\mbox{Sz}(2^r)$, $r$ any odd prime. $T$ is minimal simple, $\Out(T)=\langle \phi\ |\ o(\phi)=r\rangle$.
\item $\PSL(3,3)$. $T$ is minimal simple, $\Out(T)=\langle \gamma\ |\ o(\gamma)=2\rangle$.
\end{enumerate}
\end{thm}

If $L$ is any almost-simple group, then $L$ is primitive since it admits a maximal subgroup $M$ not containing $S=\soc(L)$, hence $M_L=\{1\}$. This is because, if $S$ was contained in all the maximal subgroups of $L$, then $S$ would be contained in the Frattini subgroup of $L$, which is nilpotent since $L$ is finite. See also \cite{BE}.

\begin{lemma}\label{G_0=S}
    Let $L$ be a minimal almost-simple group with socle $S$. Assume that $L/S$ is a cyclic $m$-group where $m$ is an odd prime number. Let $R$ be a maximal subgroup of $L$ such that $RS = L$. Then $R \cap S$ is a maximal subgroup of $S$.
\end{lemma}

\begin{proof}
Let $Y$ be a normal subgroup of $L$ containing $S$, assume that $R\cap Y$ is a maximal subgroup of $Y$ and that $Y$ is of minimal order with these properties. This makes sense because there is at least one group satisfying these properties, namely $L$. We aim to prove that $Y=S$. Both $Y/S$ and $L/S$ are (possibly trivial) cyclic $m$-groups. Moreover, $L$ can be seen as a primitive group with solvable point stabilizer $R$, hence the pair $(Y, Y\cap R)$ is listed in \cite[Tables 14-20]{LZ} with the notation $(G_0,H_0)$. Since $m$ divides $|\Out(S)|$, with the help of Theorem \ref{minimalalmostsimple} we deduce that $S\cong \PSL(2,2^m)$, $\PSL(2,3^m)$ or $Sz(2^m)$. Finally, by \cite[Tables 16, 20]{LZ}, we deduce that $Y=S$.
\end{proof}

\section{Case $m$ prime} \label{section_mprime}

In this section we will prove that if $m\neq 7$ is prime then all $m$-pyramidal groups are solvable. The proof we propose generalizes the approach taken in \cite{GG}. As usual, $K$ denotes the subgroup of $G$ generated by its involutions.

\begin{lemma}\label{ndividesa}
    Let $G$ be an $m$-pyramidal group of order $2^a \cdot d$ where $m=2^n-1 \neq 7$ is prime and $d$ is odd. If $K\cong C_2^n$, then $n$ divides $a$. 
\end{lemma}

\begin{proof}
 If $m=3$ then $n=2$ and $a \geqslant 2$, so the claim follows from \cite[Theorem 4.6]{SMG}. Now assume that $m > 3$, so that $n > 2$. Note that $m=2^n-1$ is a prime number, so $n$ is also a prime number by Lemma \ref{Mersenne}. Let $C:=C_G(K)$. Since $K\cong C_2^n$, we have $G/C\cong C_m$ or $C_m\rtimes C_n$ by Lemma \ref{el2gp}. Since $n$ is an odd prime, $G/C$ is a group of odd order. Let $Q$ be a Sylow $2$-subgroup of $C$, then $Q$ is a Sylow $2$-subgroup of $G$, so $|Q|=2^a$. By Lemma \ref{properties}(2), $N_G(Q)$ is an $m$-pyramidal group. Since $Q$ is normal in $N_G(Q)$, Lemma \ref{properties}(5) implies that $2^a = |Q| \equiv 1 \mod m$. Since the prime number $n$ is the order of $2$ in the multiplicative group $(\mathbb{Z}/m \mathbb{Z})^{\ast}$, we deduce that $n$ divides $a$.
\end{proof}

In the following proof, we will use several times the Feit-Thompson theorem, that says that any finite group of odd order is solvable. Equivalently, any nonabelian simple group has even order. We also use Burnside's Theorem \cite[Theorem 31.3]{JL}, which is the following.

\begin{thm}[Burnside] \label{Burnside}
Let $p$ be a prime number and let $r$ be an integer with $r \geqslant 1$. Suppose that $G$ is a finite group with a conjugacy class of size $p^r$. Then $G$ is not simple.
\end{thm}

The rest of this section is devoted to the proof of the following Theorem.

\begin{thm}\label{pripyra}
Let $m\neq 7$ be a prime number and let $G$ be an $m$-pyramidal group. Then $G$ is solvable.   
\end{thm}

Assume, by contradiction, that $G$ is a nonsolvable $m$-pyramidal group of minimal order. Let $O=O(G)$ be the largest normal subgroup of $G$ of odd order. Then $G/O$ is $1$-pyramidal or $m$-pyramidal by Lemma \ref{properties}(4). Set $C:=C_G(K)$.

Assume that $G/O$ is $1$-pyramidal. Then the Sylow $2$-subgroup of $G$ is a generalized quaternion group or a cyclic group, so for any two involutions $x_i$, $x_j$ of $G$ we have $x_i x_j \neq x_j x_i$. This implies that $C$ does not contain involutions, so $|C|$ is odd and hence $C$ is solvable. By Proposition \ref{dihedral} we have $K\cong D_{2m}$, therefore 
$$G/C \lesssim \Aut(D_{2m})\cong C_m\rtimes \Aut(C_m) \cong C_m\rtimes C_{m-1}.$$ 
Since both $C$ and $G/C$ are solvable, we deduce that $G$ is solvable, a contradiction.

We may assume, therefore, that $G/O$ is $m$-pyramidal. Since $O$ is solvable, the minimality of $|G|$ implies that $O = \{1\}$. Let $N$ be a minimal normal subgroup of $G$, then $N \cong S^n$ or $N \cong C^n_2$ where $n \geqslant 1$ and $S$ is a nonabelian simple group (see for example \cite[Chapter 4]{JLRB}).

Suppose first that $N\cong S^n$ where $S$ is a nonabelian simple group. Let $i(X)$ denote the number of involutions of the group $X$. Then
$$m = i(N) = (i(S)+1)^n-1 = i(S) \cdot ((i(S)+1)^{n-1}+ \ldots +1).$$
Therefore $m$ is divisible by $i(S)$, so $i(S)=m$ and $n=1$, because $m$ is a prime number and a nonabelian simple group has more than one involution. Let $x_1, x_2, \ldots, x_m$ be the $m$ involutions. Since they are conjugate in $G$ and $S \unlhd G$, we have $|x^S_i|=|x^S_j|=d$ for every $i,j$. Since $S$ has $m$ involutions, this implies that $d$ is a divisor of $m$. Since $m$ is prime and $S$ is a nonabelian simple group, $d=m$, therefore $S$ has a conjugacy class of prime size $m$, contradicting Theorem \ref{Burnside}.

\medskip

Suppose that $N\cong C_2^n$. Since $N$ contains involutions, it follows that $K=N$. Since $2^n-1=m$ is a prime number, $n$ is prime by Lemma \ref{Mersenne}. By Lemma \ref{el2gp}, $G/C$ is isomorphic to one of $C_m\rtimes C_n$, $C_m$. Assume $G/C\cong C_m\rtimes C_n$, so that there is a normal subgroup $A$ of $G$ containing $C$ such that $A/C\cong C_m$ and $G/A \cong C_n$. Note that $A$ has even order and it contains a Sylow $m$-subgroup of $G$, so Lemma \ref{properties}(3) implies that $A$ is $m$-pyramidal. The minimality of $|G|$ implies that $A$ is solvable and hence $G$ is solvable because $G/A\cong C_n$, giving a contradiction.

So from now on we can assume that $K=N \cong C_2^n$ and $G/C\cong C_{2^n-1} = C_m$. Note that both $K$ and $C$ are normal in $G$. We will use several times the following fact: if $H$ is a proper subgroup of $G$ such that $HC=G$ then $H$ is solvable. When $|H|$ is odd, this follows from the Feit-Thompson theorem, and when $|H|$ is even, this follows from Lemma \ref{properties}(1) and the minimality of $|G|$. Since $G/C\cong C_m$ has prime order, $C$ is a maximal normal subgroup of $G$. Since $O(G)=\{1\}$, the minimal normal subgroups of $G$ have even order, so they contain all the involutions of $G$. This implies that the unique minimal normal subgroup of $G$ is $K$. 

\subsection{Step 1} We will prove that there exists a normal subgroup $N$ of $G$ with $\Phi(G) < N \leqslant C$ such that $G/N$ is a cyclic $m$-group, $N/\Phi(G)$ is a nonabelian chief factor of $G$, isomorphic to $S^t$ for some nonabelian simple group $S$, and $\Phi(G)$ is a $2$-group containing $K$. 
$$\xymatrix{\{1\} \ar@{-}[rr]_{2^a} & & \Phi(G) \ar@{-}[rr]_{S^t} & & N \ar@{-}[rr]_{m^b} & & C \ar@{-}[rr]_{m} & & G}$$
\begin{proof}
Let $N$ be a normal subgroup of $G$ such that $G/N$ is cyclic and $N$ is of minimal order with this property. If $G/N$ is not an $m$-group, then there exists $L/N \unlhd G/N$ such that $|G:L|$ is a prime distinct from $m$, so that $LC=G$ since $|G:C|=m$, so $L$ is solvable. Since $G/L$ is solvable, this contradicts the fact that $G$ is nonsolvable. So $G/N$ is a cyclic $m$-group. Moreover $N$ is contained in $C$ because otherwise $NC=G$ and then $N$ would be solvable, so $G$ would be solvable as well. 
Consider a normal subgroup $R$ of $G$ contained in $N$ with the property that $N/R$ is a minimal normal subgroup of $G/R$. We claim that $R=\Phi(G)$. At this step, we need to prove that $\Phi(G)$ is a nontrivial $2$-group. First, we assume that $\Phi(G)=\{1\}$, then $K\cap \Phi(G)=\{1\}$. Since $K$ is an abelian normal subgroup of $G$, $K$ has a complement $A$ in $G$, that is $G=K\rtimes A$. The subgroup $K$ is generated by all involutions of $G$, so $|A|$ is odd, this implies that $G$ is solvable, a contradiction. Next, note that $\Phi(G)$ is a $2$-group because otherwise there would exist a nontrivial Sylow subgroup $P$ of $\Phi(G)$ of odd order, and since $P \unlhd_c \Phi(G) \unlhd G$, being $\Phi(G)$ nilpotent, $P$ would be a nontrivial normal subgroup of $G$ of odd order, contradicting $O(G)=\{1\}$. In particular $\Phi(G)$ contains involutions, so $K \leqslant \Phi(G)$. Since $\Phi(G)/\Phi(G) \cap N \cong \Phi(G)N/N \leqslant \Phi(G/N)$ is an $m$-group, we deduce that $\Phi(G) \leqslant N$. By minimality of $|N|$, $G/R$ is not a cyclic $m$-group, so if $x \in G$ is an $m$-element that generates $G/C$, then $R \langle x \rangle$ is a proper subgroup of $G$, moreover $R \langle x \rangle C = G$, so $R$ is solvable. Since $R \leqslant N \leqslant C$, to prove that $R \leqslant \Phi(G)$ it is enough to prove that if $M$ is a maximal subgroup of $G$ distinct from $C$, then $R \leqslant M$. We have $MC=G$, so $M$ is solvable hence $MR \neq G$ being $G$ nonsolvable and $M,R$ solvable. Therefore $R \leqslant M$, implying that $R \leqslant \Phi(G)$. Since $G/N$ is an $m$-group and $G$ is nonsolvable, we deduce that $\Phi(G) \neq N$, in other words $\Phi(G)$ is properly contained in $N$. Since $R \leqslant \Phi(G) < N$ and $N/R$ is a minimal normal subgroup of $G/R$, we deduce that $R=\Phi(G)$. Since $\Phi(G)$ is nilpotent, the nonsolvability of $G$ implies that $N/\Phi(G)$ is a nonabelian chief factor of $G$, isomorphic to a direct power $S^t$ for some nonabelian simple group $S$.
\end{proof}

\subsection{Step 2} Let $M$ be a maximal subgroup of $G$ with $M \neq C$. Then the normal core $M_G$ of $M$ in $G$ equals $\Phi(G)$ and consequently $G/\Phi(G)$ is a primitive group.

\begin{proof}
By Lemma \ref{properties}(1), the maximal subgroups of $G$ distinct from $C$ are solvable by minimality of $|G|$, since they supplement $C$. If $M_1$ is a maximal subgroup of $G$ such that $M_1\neq C$ then $M_G \leqslant M_1$, since otherwise we would have
$M_1M_G=G$, contradicting the fact that $G$ is nonsolvable, since $M_1$ and $M_G$ are solvable. So
every maximal subgroup of $G$ different from $C$ contains $M_G$, implying that $\Phi(G) = M_G\cap C$. To conclude the proof, we need to prove that $M_G \leqslant C$. If this is not the case, then $C M_G = G$.
Note that $G/M_G=CM_G/M_G\cong C/\Phi(G)$ has a subgroup $N/\Phi(G)\cong S^t$. Since every nonabelian simple group has order divisible by $4$, there exist two subgroups $A, B\leqslant G$ containing $M_G$
such that $A\leqslant B$ and $|B:A|=|A:M_G|=2$. Since $A$ and $B$ contain $M_G$, $AC=BC=G$, hence
$A$ and $B$ are both $m$-pyramidal by Lemma \ref{properties}(1). Since $K \leqslant \Phi(G) \leqslant M_G \leqslant A$ and $K\cong C^n_2$, this contradicts Lemma \ref{ndividesa}.
\end{proof}

\subsection{Step 3} Let $W:=G/\Phi(G)$, it is a primitive group whose socle is $N/\Phi(G) \cong S^t$, it is the unique minimal normal subgroup of $W$. Let $T_1$ be the first direct factor of $N/\Phi(G) \cong S^t$ and let $G_1 := N_W(T_1)/C_W(T_1)$. Then $G_1$ is an almost-simple group with socle $T_1 C_W(T_1)/C_W(T_1) \cong S$ and we will identify $G_1$ with a subgroup of $\Aut(S)$ containing $S$. We say that a finite group $G$ of order $2^ad$ with $d$ odd is of \textit{$n$-type} if $n$ divides $a$. Then $W$ is of $n$-type and, if $U$ is a maximal subgroup of $G_1$ such that $US=G_1$, then $U$ is a solvable group of $n$-type.

\begin{proof}
Write $|\Phi(G)|=2^a$ and $|G|=2^b\cdot d$, we have $|W|=2^{b-a}\cdot d$. By Lemma \ref{properties}(5), $m=2^n-1$ divides $2^a-1$, and since $n$ is prime we deduce that $n$ divides $a$. Moreover $n$ divides $b$ by Lemma \ref{ndividesa}, therefore $n$ divides $b-a$, in other words $W$ is of $n$-type. We refer to \cite[Chapter 1]{BE} for the general properties of primitive groups. Since $N/\Phi(G)$ is a nonabelian minimal normal subgroup of $W$ and $G/N$ is cyclic, $W$ is a primitive group of type II, meaning that it admits a unique minimal normal subgroup and this subgroup is nonabelian. Since $G/N$ is a cyclic $m$-group, we have an embedding of $W$ in the standard wreath product $G_1 \wr P$ where $P$ is a cyclic $m$-subgroup of $\Sym(t)$, the image of the permutation representation of $W$ acting transitively by conjugation on the $t$ minimal normal subgroups of its socle (see \cite[Remark 1.1.40.13]{BE}). In particular, $t$ is a power of $m$.

Let $U$ be a maximal subgroup of $G_1$ such that $US=G_1$ and let $V := (U \cap S)^t \leqslant S^t = \soc(W)$. We claim that $N_W(V) \soc(W) = W$. Let $g \in W$. Identifying $W$ with a subgroup of $G_1 \wr P$, we can write $g = (x_1,\ldots, x_t) \gamma$ where $x_i \in G_1$ for all $i$ and $\gamma\in P$. Since $US=G_1$, there exist $s_i \in S$, $u_i\in U$ such that $x_i = s_iu_i$ for all $i=1,\ldots,t$, and setting $n:=(s_1,\ldots, s_t) \in \soc(W)$ we have $h := (u_1,\ldots, u_t)\gamma = n^{-1} g \in W$, therefore
$$V^h = ((U \cap S)^{u_1} \times \cdots \times (U \cap S)^{u_t})^\gamma = V^\gamma = V.$$
We deduce that $h=n^{-1}g\in N_W(V)$ and since $n \in \soc(W)$ the claim follows.

We claim that $N_W(V) \neq W$. If this was not the case, then $V$ would be normal in $W$ and, since $N/\Phi(G)$ is a minimal normal subgroup of $W$ and $U \cap S \neq S$, since $US=G_1$, the simplicity of $S$ implies $U \cap S=\{1\}$ and $G_1/S \cong U$. It follows that $U$ is isomorphic to a section of $W/\soc(W) \cong G/N$, which is cyclic, therefore $U$ is a cyclic maximal subgroup of the almost-simple group $G_1$, contradicting Herstein's theorem, which says that a nonsolvable finite group cannot have abelian maximal subgroups (see \cite[Theorem 5.53]{AM}). So the claim follows.

Since $\Phi(G)$ is a nontrivial $2$-group, the above two paragraphs imply that the preimage of $N_W(V)$ in $G$ is a proper subgroup of $G$ of even order and it is not contained in $C$, so it is $m$-pyramidal by Lemma \ref{properties}(1) hence it is solvable by minimality of $|G|$. So $U \cap S$ is solvable. Since $U/U \cap S \cong US/S = G_1/S$ is a cyclic $m$-group, we deduce that $U$ is solvable.

Summarizing, we have $N_W(V) \soc(W) = W$, $N_W(V) \neq W$, $G_1/S$ is a cyclic $m$-group and $G_1$ has only one nonsolvable maximal subgroup, which is the unique maximal subgroup containing $S$. By Lemma \ref{G_0=S}, $U \cap S$ is a maximal subgroup of $S$. Therefore, $N_S(U\cap S)=U\cap S$, hence
$$N_W(V)\cap soc(W)=N_{soc(W)}(V)=N_{s^m}((U\cap S)^m)=(N_S(U\cap S))^m=V.$$
Since $N_W(V) \soc(W) = W$ and $U/U \cap S \cong G_1/S \cong C_{m^s}$ is a cyclic $m$-group, writing $|U|=2^d \cdot r$ with $r$ odd and $W/\soc(W) \cong C_{m^b}$, we have 
\begin{align*}
|N_W(V)| = |W:\soc(W)| \cdot |N_W(V) \cap \soc(W)| = m^b \cdot |U \cap S|^t = m^{b-st} \cdot 2^{dt} \cdot r^t.  
\end{align*}
Since the preimage of $N_W(V)$ in $G$ is $m$-pyramidal by Lemma \ref{properties}(1), and $\Phi(G)$ is of $n$-type by Lemma \ref{properties}(5), $N_W(V)$ is of $n$-type by Lemma \ref{ndividesa}, in other words $dt$ is divisible by $n$. Since $t$ is a power of $m$, we conclude that $d$ is divisible by $n$. This means exactly that $U$ is of $n$-type.
\end{proof}  

\subsection{Step 4} $G_1=S$.

\begin{proof}
Recall that, if $X$ is an almost-simple group with socle $S$, which is nonabelian and simple, we say that a maximal subgroup $H$ of $S$ is $X$-\textit{ordinary} if its $X$-class equals its $S$-class, in other words for every $x\in X$ there exists $s\in S$ such that $H^x=H^s$. If $H$ is an $X$-ordinary maximal subgroup of $S$, then $N_X(H)$ is maximal in $X$ with $N_X(H)S=X$, and $|N_X(H)|=|X||H|/|S|$ by \cite[Lemma 2.2]{GG}. Of course, if $S \leqslant X \leqslant Y \leqslant \Aut(S)$ and $H$ is a $Y$-ordinary subgroup of $S$, then $H$ is also $X$-ordinary. Concerning our situation, we have $S\leqslant G_1\leqslant \Aut(S)$, $G_1/S$ is a cyclic $m$-group and every maximal subgroup of $G_1$ not containing $S$ is solvable. Assume by contradiction that $G_1\neq S$. Since $G_1/S$ is not a $2$-group, $\Out(S)$ is not a $2$-group, so with the help of Theorem \ref{minimalalmostsimple} we have $S\cong \PSL(2,2^m)$, $\PSL(2,3^m)$ or $Sz(2^m)$. By \cite[Table 8.1]{BHR}, if $S\cong \PSL(2,2^m)$ or $\PSL(2,3^m)$ then $S$ has a $G_1$-ordinary maximal subgroup $H$ of type $E_{2^m}:(2^m-1)$ in the first case and $\PSL(2,3)$ in the second case when $m>3$. Thus $|N_{G_1}(H)|=2^m\cdot m\cdot (2^m-1)$ or $2^2\cdot 3\cdot m$ because $|G_1:S|=m$, contradicting the fact that $N_{G_1}(H)$ is of $n$-type since $m$ and $n$ are distinct prime numbers, with $n$ odd unless $(n,m)=(2,3)$. If $m=3$ and $S \cong \PSL(2,3^3)$ then $G_1$ has a maximal subgroup of type $C_{13} \rtimes C_6$, which is not of $2$-type (see the second line of \cite[Table 8.1]{BHR}). If $S\cong Sz(2^m)$, by \cite[Table 8.16]{BHR} $S$ has a $G_1$-ordinary maximal subgroup $H$ of type $D_{2(2^m-1)}$. It is easy to deduce that $|N_{G_1}(H)|=2\cdot m\cdot(2^m-1)$ and $N_{G_1}(H)$ is not of $n$-type, a contradiction.
\end{proof}  
  
\subsection{Step 5} Since $W$ is a subgroup of $S\wr P$ containing $S^t = \soc(W)$, with $W$ projecting
surjectively onto the transitive cyclic group $P \leqslant \Sym(t)$, we deduce that $W$ is isomorphic to the standard wreath product $S\wr P$, where $P$ acts on $S^t$ by permuting the coordinates. Consider
$$\Delta=\{(s, s, \ldots, s)\ :\ s\in S\} \leqslant S^t, \hspace{1cm} H:=N_W(\Delta).$$
It is clear that $P \leqslant H$, hence $\soc(W)H=W$. Since $W$ has a normal subgroup of index $m$, we have $t\geqslant 2$, therefore $\Delta$ is not normal in $W$, since $\soc(W)$ a minimal normal subgroup of $W$, therefore $H \neq W$. Since $\Delta \cong S$ is nonsolvable, we obtain that the preimage of $H$ in $G$ is a nonsolvable proper subgroup of $G$, of even order and supplementing $C$, so it is $m$-pyramidal by Lemma \ref{properties}(1). This contradicts the minimality of $|G|$.

This concludes the proof of the fact that, if $m\neq 7$ is a prime number, then any $m$-pyramidal group is solvable.

\section{Proof of Theorem \ref{maintheorem}} \label{primepower}

In this section, we will prove Theorem \ref{maintheorem}.

\begin{prop} \label{Npyr}
Let $N$ be a finite abelian group of odd order and let $A$ be a subgroup of $\Aut(N)$ containing the inversion map $\iota:N \to N$, $n \mapsto n^{-1}$ and with the property that $\iota$ is the unique element of order $2$ in $A$. Then the semidirect product $N \rtimes A$ is $|N|$-pyramidal.
\end{prop}
\begin{proof}
Write $G=NA$. If an element $na \in NA = G$ has order $2$, then $1 = (na)^2 = nana = (n \cdot n^{a^{-1}}) \cdot a^2$ hence $a^2=1$ and $n^a = n^{-1}$. Since $|N|$ is odd, $n^a = n^{-1} \neq n$, therefore $a \neq 1$. This implies that $a$ has order $2$, so $a=\iota$. This proves that the involutions of $G$ are precisely the elements of the form $n \iota$ with $n \in N$ arbitrary, so $G$ has $|N|$ involutions. Moreover $\iota^n = n^{-1} \iota n = n^{-2} \iota$. Since $|N|$ is odd, we deduce that all the involutions of $G$ are conjugate to $\iota$, hence $G$ is $|N|$-pyramidal.
\end{proof}

For example, let $q$ be an odd prime power and let $V=\mathbb{F}_q^2$. The group $G=V \rtimes \SL(2, q)$ is $q^2$-pyramidal by the above proposition. This shows that, if $p$ is a prime and $k$ is an even positive integer, with $p^k \neq 9$, then there exists a nonsolvable $p^k$-pyramidal group. More in general, if there exist $N$ and $A$ as in the proposition, with $A$ nonsolvable, then we can construct a nonsolvable $|N|$-pyramidal group. 

For example, if there exists a nonsolvable subgroup $H$ of $\GL(k,q)$ with $k$ odd and $-1$ is the unique involution of $H$ then the above argument shows that there exists a nonsolvable $q^k$-pyramidal group. However, this $H$ does not exist, because assuming it exists, defining $U:=H \cap \SL(k,q)$, the order $|U|$ is odd because $H$ contains $-1$ as unique involution and, since $k$ is odd, $-1 \not \in \SL(k,q)$. Therefore $U$ is solvable and $H/U$ is abelian because it is isomorphic to a subgroup of $\GL(k,q)/\SL(k,q) \cong \mathbb{F}_q^{\ast}$, so $H$ is solvable.

We now prove Theorem \ref{maintheorem}, which we state again for convenience.

\begin{thm}
Let $m$ be a prime power $p^k$ with $p$ an odd prime, $p \neq 7$. Then the following are equivalent.
\begin{enumerate}
    \item Every $m$-pyramidal group is solvable.
    \item $k$ is odd or $m=9$.
\end{enumerate}
\end{thm}

If $k$ is even and $m \neq 9$ then, choosing $q=p^{k/2}$, the group $\mathbb{F}_q^2 \rtimes \SL(2, q)$ is $m$-pyramidal and nonsolvable, as we observed above. This proves that (1) implies (2). The proof of the other implication is the content of the following two propositions. We will use the notation $i(G)$ to denote the number of involutions of $G$.

\begin{prop}
Any $9$-pyramidal group is solvable.
\end{prop}

\begin{proof}
We prove the result by contradiction. Assume $G$ is a nonsolvable $9$-pyramidal group of minimal order. Let $O(G)$ be the maximal normal subgroup of $G$ with odd order. Then $G/O(G)$ is a $1$, $3$ or $9$-pyramidal group. Obviously, $G/O(G)$ is not a $3$-pyramidal group because otherwise $G/O(G)$ would be solvable and then $G$ would be solvable as well, a contradiction.  

Suppose $G/O(G)$ is a $9$-pyramidal group. The minimality of $|G|$ implies that $O(G)=\{1\}$. Let $N$ be a minimal normal subgroup of $G$, then $N \cong C^n_2$ or $S^n$ where $n\geqslant 1$ and $S$ is a nonabelian simple group. Since all involutions of $G$ are conjugate, $2^n-1=9$ or $i(N)=(i(S)+1)^n-1=9$, this implies that $n=1$ and $N \cong S$, and so $G$ is an almost simple group with socle $S$. Let $x_1, x_2,\ldots, x_9 \in S$ be the $9$ involutions of $G$. Since they are conjugate in $G$, the conjugacy classes $x_i^S$, $i=1,\ldots,9$ all have the same size $d$. This implies that $d \in \{3,9\}$ because $Z(S)=\{1\}$, contradicting Theorem \ref{Burnside}.

Finally, we assume that $G/O(G)$ is $1$-pyramidal. Then $G/O(G)$ is nonsolvable and the Sylow $2$-subgroup of $G/O(G)$ is a generalized quaternion group. Therefore, the Sylow $2$-subgroup of $G$ is also a generalized quaternion group and hence each one of them contains a unique involution. This implies that, if $x$, $y$ are two arbitrary distinct involutions of $G$, we have $xy \neq yx$. It follows that $|C_G(K)|$ is odd, and the Sylow $2$-subgroup of $J:=G/C_G(K)$ is also a generalized quaternion group. On the other hand, $J$ is isomorphic to a transitive subgroup of $S_9$ and $J$ does not contain $A_9$ because $A_9$ and $S_9$ do not have generalized quaternion Sylow $2$-subgroups. The maximal imprimitive subgroups of $S_9$ are isomorphic to the wreath product $S_3 \wr S_3$, which is solvable, therefore $J$ is primitive of degree $9$. The nonsolvable primitive groups of degree $9$ do not have generalized quaternion Sylow $2$-subgroups, this can be easily seen using the \text{AllPrimitiveGroups} library in GAP \cite{GAP4}. We have reached a contradiction.
\end{proof}

\begin{prop}
Let $G$ be a $p^k$-pyramidal group, where $p\neq 7$ is prime and $k$ is odd. Then $G$ is solvable.
\end{prop}
\begin{proof}
We prove the result by contradiction. Assume $G$ is a nonsolvable pyramidal group with a number of involutions that is a prime power with odd exponent, and assume $G$ has minimal order with these properties. Write $m=p^k$ with $k$ odd. Let $O:=O(G)$ be the largest normal subgroup of $G$ with odd order, then $G/O$ is a $p^a$-pyramidal group, where $0\leqslant a\leqslant k$. Let $N/O$ be a minimal normal subgroup of $G/O$, then $N/O \cong S^n$ or $C^n_2$, where $n\geqslant 1$ and $S$ is a nonabelian simple group.  First, if $N/O\cong S^n$ then $i(G/O)=(i(S)+1)^n-1=p^a$. Since $i(S) > 1$ and $p$ is odd, Lemma \ref{Mersenne} implies that $n=1$, so $G/O$ is an almost simple group with socle $S$ and $i(S)=p^a$. Since the involutions are all conjugate, they all belong to $S$ and their conjugacy classes in $S$ all have the same size. This implies that the conjugacy class size of an involution in $S$ is a prime power, contradicting Theorem \ref{Burnside}.

In the rest of the proof we will assume that $N/O \cong C^n_2$. In particular $2^n-1 = p^a$ hence $a \in \{0,1\}$ by Lemma \ref{Mersenne}, that is, $G/O$ is $p$-pyramidal or $1$-pyramidal. If $G/O$ is $p$-pyramidal then, by Theorem \ref{pripyra}, $G/O$ is solvable, and hence $G$ is solvable, a contradiction. In the rest of the proof we will assume that $G/O$ is $1$-pyramidal, so that the Sylow $2$-subgroups of $G/O$ are generalized quaternion groups. This implies that the Sylow $2$-subgroups of $G$ are also generalized quaternion, hence if $x,y$ are two arbitrary involutions of $G$ then $xy \neq yx$. By Lemma \ref{quaternionsylow}, the center $Z/O$ of $G/O$ has order $2$. Then the Sylow $2$-subgroups of $G/Z$ are dihedral, hence $G/Z$ satisfies one of the items of Lemma \ref{dihedralsylow}, in particular either $G/Z \cong A_7$ or $G/Z$ has a simple normal subgroup $S/Z \cong \PSL(2,q)$, $q=r^f$, $r$ an odd prime. In the second case $G/Z$ is an extension of $S/Z$ and a subgroup of $\Out(\PSL(2,q)) \cong C_2 \times C_f$. Observe that the centralizer $H:=C_{G}(\varepsilon)$ has $\varepsilon$ as unique involution because any two involutions of $G$ do not commute, and hence the order of $C=C_G(K)$ is odd and $C$ is contained in $O$. Obviously, $|G:H|=m=p^k$. Now, 
$$|S:HZ \cap S| = |HS:HZ| = \frac{|HS:H|}{|HZ:H|}$$ 
hence $HZ \cap S$ has prime power index in $S$, say $p^b$ with $b \leqslant k$. 

We claim that $H$ is nonsolvable. If $HZ\cap S=S$, i.e., $S\leqslant HZ$, then $HZ/Z \cong H/H \cap Z$ is nonsolvable since it contains $S/Z$, so $H$ is nonsolvable. If $HZ\cap S\neq S$, then \cite[Theorem 1]{RMG} implies that either $HZ \cap S$ has index $r^f+1 = p^b$ in $S$, contradicting the fact that both $r$ and $p$ are odd primes, or $S/Z \cong A_7$ with $p^b=7$, or $S/Z \cong \PSL(2,11)$ with $p^b=11$. The following quotient group is isomorphic to $A_6$ or $A_5$, respectively.
$$\frac{HZ \cap S}{Z} = \frac{Z(S \cap H)}{Z} \cong \frac{S \cap H}{S \cap H \cap Z} = \frac{S \cap H}{Z \cap H}$$
It follows that $H$ is nonsolvable.

Let $R$ be a minimal normal subgroup of $G$ contained in $O$. Note that $R$ is an elementary abelian group. Let
$$I_{\varepsilon} := \{n \in R\ :\ n^{\varepsilon} = n^{-1}\} \leqslant R.$$ 
Lemma \ref{properties}(4) implies that $G/R$ is $m/\ell$-pyramidal where $\ell=|I_{\varepsilon}| = p^t$. By the minimality of $|G|$ we have that $k-t$ is even, thus $t$ is odd. $I_{\varepsilon}$ is normalized by $H$ because, if $h \in H$ and $n \in I_{\varepsilon}$, then $h \varepsilon = \varepsilon h$ hence
$$(n^h)^{\varepsilon} = n^{h \varepsilon} = n^{\varepsilon h} = (n^{-1})^h = (n^h)^{-1}.$$
We can see $R$ as a finite dimensional $\mathbb{F}_p$-vector space acted upon by the linear transformation $\varepsilon$. Since $\varepsilon^2=1$ and $p$ is odd, $\varepsilon$ is diagonalizable over $\mathbb{F}_p$ and its unique eigenvalues are $1$ and $-1$. Observe that $I_{\varepsilon}$ is precisely the eigenspace of $-1$ in $R$ and $H \cap R$ is the eigenspace of $1$, so 
$$\ell = p^t = |I_{\varepsilon}| = |R:H \cap R| = |HR:H|.$$
Since $|G:H|$ is odd, any Sylow $2$-subgroup of $HR$ is a Sylow $2$-subgroup of $G$, hence $HR$ is a $p^t$-pyramidal group by Lemma \ref{quatpyr}. Since $H$ is nonsolvable, using the minimality of $|G|$ again, we have $HR=G$ and $t=k$. Observe that $R$ is an abelian minimal normal subgroup of $G$, thus $H\cap R$ is normal in $H$ and in $R$. Since $HR=G$, we deduce that $H \cap R$ is normal in $G$ hence $H\cap R=\{1\}$ by the minimality of $R$. It follows that $G=R\rtimes H$, $R\cong \mathbb{F}_p^k$ and $\varepsilon$ acts on $R$ as the inversion map $x \mapsto -x$. Let $\overline{H} := H/C_H(R)$. Since $\varepsilon$ is the only involution of $H$ and it does not centralize $R$, the centralizer $C_H(R)$ has odd order, therefore $\overline{H}$ is $1$-pyramidal by Lemma \ref{properties}(4). We can identify $\overline{H}$ with a subgroup of $\GL(k,p)$. Let $U := \overline{H} \cap \SL(k,p)$. If $|U|$ is odd, then $U$ is solvable and hence $\overline{H}$ is also solvable, because $\GL(k,p)/\SL(k,p)$ is cyclic, so $H$ is solvable, a contradiction. Therefore $|U|$ is even, hence $U$ contains an involution. Since $\overline{\varepsilon} = \varepsilon C_H(R)$ is the only involution of $\overline{H}$, we have $\overline{\varepsilon} \in U$, so that $\det(\overline{\varepsilon})=1$. Since $\overline{\varepsilon}$ is the inversion map, we deduce that $(-1)^k = \det(\overline{\varepsilon}) = 1$ contradicting the fact that $k$ is odd.
\end{proof}

\section{Proof of Theorem \ref{orders}} \label{section_orders}

In this section we will prove Theorem \ref{orders}, which we now state again for convenience.

\begin{thm}
Let $m$ be an odd prime number distinct from $7$ and let $X_m$ be the set of orders of $m$-pyramidal groups. If $m$ has the form $2^n-1$ set $Y_m = \{2^a \cdot m \cdot d\ :\ n|a,\ d\ \mbox{odd}\}$, otherwise $Y_m = \varnothing$. Write $m-1=2^t \cdot r$ with $r$ odd and let $Z_m = \{2^a \cdot m \cdot d\ :\ 1 \leqslant a \leqslant t,\ d\ \mbox{odd}\}$. Then $X_m = Y_m \cup Z_m$.
\end{thm}

\begin{proof}
If $P$ is any nontrivial $2$-subgroup of $\Aut(C_m) \cong C_{m-1}$ then, by Proposition \ref{Npyr}, $C_m \rtimes P$ is $m$-pyramidal of order $m \cdot 2^a$ with $1 \leqslant a \leqslant t$, and if $d$ is any odd positive integer then $C_d \times (C_m \rtimes P)$ is $m$-pyramidal of order $2^a \cdot m \cdot d$, therefore $Z_m \subseteq X_m$. If $m$ has the form $2^n-1$ then the multiplicative group of the finite field $F=\mathbb{F}_{2^n}$ is cyclic generated by an element $x$ of order $m$. The multiplication by $x$ induces an automorphism $\psi$ of the additive group $F$, also of order $m$, and $\langle \psi \rangle$ acts transitively on $F \setminus \{0\}$. This automorphism $\psi$ is usually called Singer cycle. Consider a homocyclic group $H_{l,n} = (\mathbb{Z}/2^l\mathbb{Z})^n$. Its Frattini subgroup $\Phi$ is isomorphic to $H_{l-1,n}$. We denote by $\GL(n,\mathbb{Z}/2^l\mathbb{Z})$ the group of $n \times n$ matrices with coefficients in $\mathbb{Z}/2^l\mathbb{Z}$ and invertible determinant. The natural map $\GL(n,\mathbb{Z}/2^l\mathbb{Z}) \to \GL(n,\mathbb{Z}/2\mathbb{Z})$,
given by componentwise reduction modulo $2$, is surjective and its kernel has size $2^{n^2 (l-1)}$, so there exists $\tau \in \Aut(H_{l,n})$ inducing the Singer cycle automorphism $\psi$ on $H_{l,n}/\Phi \cong H_{1,n} = (\mathbb{Z}/2\mathbb{Z})^n$ and the order of $\tau$ is $m$ multiplied by a power of $2$. Raising $\tau$ to a suitable power of $2$ gives an automorphism of $H_{l,n}$ of order $m$, call it $\gamma$. Observe that the only fixed point of $\gamma$ in its action on $H_{l,n}/\Phi$ is the trivial element. Let $K \cong C_2^n$ be the subgroup of $H_{l,n}$ generated by the involutions, and let $\varepsilon \in K$ be an involution. Then there exists $a \in H_{l,n}$ such that $\varepsilon = a^{2^{l-1}}$. If $\gamma$ fixes $\varepsilon$, then
$$(\gamma(a) \cdot a^{-1})^{2^{l-1}} = \gamma(a^{2^{l-1}}) \cdot (a^{2^{l-1}})^{-1} = \gamma(\varepsilon) \cdot \varepsilon^{-1} = 1,$$
therefore $\gamma(a) a^{-1} \in \Phi$, in other words $\gamma(a) \Phi = a \Phi$. This is a contradiction, since we know that $\gamma$ acts fixed point freely on $H_{l,n}/\Phi$. Therefore the action of $\gamma$ on $K$ is nontrivial, and since $m$ is a prime number, this implies that $\langle \gamma \rangle$ acts transitively on $K \setminus \{1\}$. So $H_{l,n} \rtimes \langle \gamma \rangle$ is $m$-pyramidal of order $2^{nl} \cdot m$. If $d$ is any odd positive integer, the direct product $C_d \times (H_{l,n} \rtimes \langle \gamma \rangle)$ is $m$-pyramidal of order $2^{nl} \cdot m \cdot d$. This proves that $Y_m \subseteq X_m$.

We are left to prove that $X_m \subseteq Y_m \cup Z_m$. Let $G$ be an $m$-pyramidal group. Since $m$ is the index of the centralizer of an involution, we can write $|G| = 2^a \cdot m \cdot d$ with $d$ odd. We will prove that $|G| \in Y_m \cup Z_m$ by induction on $|G|$ (for fixed $m$). Let $N$ be a minimal normal subgroup of $G$. If $|N|$ is even, then $N$ is an elementary abelian $2$-group $N \cong C_2^n$ and $m=2^n-1$. Lemma \ref{ndividesa} implies that $n$ divides $a$. Now assume $|N|$ is odd. We know that $G/N$ is $1$-pyramidal or $m$-pyramidal (because $m$ is prime) and if it is $m$-pyramidal we conclude by induction. So now assume that $G/N$ is $1$-pyramidal. Since $|N|$ is odd, the Sylow $2$-subgroups of $G$ are either cyclic or generalized quaternion, so $xy \neq yx$ for every two involutions $x,y$ of $G$. If $\varepsilon$ is an involution of $G$, then $\varepsilon N$ is the unique involution of $G/N$ hence the subgroup $N \langle \varepsilon \rangle$ of $G$ is normal in $G$ and it contains involutions, hence the subgroup of $G$ generated by the involutions is $K = N \langle \varepsilon \rangle \cong D_{2m}$ by Proposition \ref{dihedral}, therefore $N \cong C_m$. This also implies that $|C_G(N)|$ is odd. Since $G/C_G(N)$ is isomorphic to a subgroup of $\Aut(N) \cong C_{m-1}$, it follows that $a \leqslant t$.
\end{proof}

\section{Acknowledgements}
The first author acknowledges the support of the CAPES PhD
fellowship and the NSF of China - Grant number 12161035. 
The second author acknowledges the support of Conselho Nacional 
de Desenvolvimento Cient\'ifico e Tecnol\'ogico (CNPq), Universal
- Grant number 402934/2021-0.

\bibliographystyle{unsrt}
\small\bibliography{reference_R}

\begin{thebibliography}{10}

\bibitem{Kirkman}
T.~P. Kirkman.
\newblock On a problem in combinations.
\newblock {\em Cambridge and Dublin Mathematical Journal}, 2:191--204, 1847.

\bibitem{RW}
D.~K. Ray-Chaudhuri and R.~M. Wilson.
\newblock Solution of {K}irkman's {S}choolgirl {P}roblem.
\newblock {\em Proceedings of Symposia in Pure Mathematics}, 19:187--203, 1971.

\bibitem{BRT}
M.~Buratti, G.~Rinaldi, and T.~Traetta.
\newblock 3-pyramidal {S}teiner triple systems.
\newblock {\em Ars Mathematica Contemporanea}, 13:95--106, 2017.

\bibitem{SMG}
S.~Bonvicini, M.~Buratti, M.~Garonzi, G.~Rinaldi, and T.~Traetta.
\newblock The first families of highly symmetric {K}irkman {T}riple {S}ystems
  whose orders fill a congruence class.
\newblock {\em Designs, Codes and Cryptography}, 89(12):2725--2757, 2021.

\bibitem{GG}
X.~Gao and M.~Garonzi.
\newblock The structure of $3$-pyramidal groups.
\newblock {\em Journal of Algebra}, 636:75--87, 2023.

\bibitem{MI}
I.~M. Isaacs.
\newblock {\em Finite Group Theory. Graduate Studies in Mathematics, 92}.
\newblock American Mathematical Society, 2008.

\bibitem{HY}
H.~Yamaki.
\newblock The order of a group of even order.
\newblock {\em Proceedings of the American Mathematical Society},
  136(2):397--402, 2008.

\bibitem{DanLevy}
D.~Levy.
\newblock Characterization of the solvable radical by {S}ylow multiplicities.
\newblock {\em Journal of Algebra}, 635:23--84, 2023.

\bibitem{Suzuki}
R.~Brauer and M.~Suzuki.
\newblock On finite groups of even order whose 2-{S}ylow group is a quaternion
  group.
\newblock {\em Proceedings of the National Academy of Sciences of the United
  States of America}, 45(12):1757--1759, 1959.

\bibitem{BE}
A.~Ballester-Bolinches and L.~M. Ezquerro.
\newblock {\em Classes of Finite Groups}.
\newblock Springer, 2006.

\bibitem{Daniel}
D.~Gorenstein and J.~H. Walter.
\newblock The characterization of finite groups with dihedral {S}ylow
  2-subgroups. \mbox{I}.
\newblock {\em Journal of Algebra}, 2(1):85--151, 1965.

\bibitem{KZ}
K.~Zsigmondy.
\newblock Zur theorie der potenzreste.
\newblock {\em Monatshefte f\"{u}r Mathematik und Physik}, 3:265--284, 1892.

\bibitem{PJC}
P.~J. Cameron.
\newblock {\em Permutation Groups}.
\newblock Cambridge University Press, 2010.

\bibitem{UD}
U.~Dempwolff.
\newblock On the second cohomology of $\mbox{GL}(n,2)$.
\newblock {\em Journal of the Australian Mathematical Society}, 16(2):207--209,
  1973.

\bibitem{LZ}
C.~H. Li and H.~Zhang.
\newblock The finite primitive groups with soluble stabilizers, and the
  edge-primitive s-arc transitive graphs.
\newblock {\em Proceedings of the London Mathematical Society},
  103(3):441--472, 2011.

\bibitem{JL}
G.~James and M.~Liebeck.
\newblock {\em Representations and Characters of Groups}.
\newblock Cambridge University Press, 2001.

\bibitem{JLRB}
J.~L. Alperin and R.~B. Bell.
\newblock {\em Groups and Representations}, volume 162.
\newblock Springer, 1995.

\bibitem{AM}
A.~Mach{\`\i}.
\newblock {\em Groups: An Introduction to Ideas and Methods of the Theory of
  Groups}.
\newblock Springer, 2012.

\bibitem{BHR}
J.~N. Bray, D.~F. Holt, and C.~M. Roney-Dougal.
\newblock {\em The maximal subgroups of the low-dimensional finite classical
  groups}, volume 407.
\newblock Cambridge university press, 2013.

\bibitem{GAP4}
The GAP~Group.
\newblock {\em {GAP -- Groups, Algorithms, and Programming, Version 4.13.1}},
  2024.

\bibitem{RMG}
R.~M. Guralnick.
\newblock Subgroups of prime power index in a simple group.
\newblock {\em Journal of Algebra}, 81:304--311, 1983.

\end{thebibliography}

\end{document}